\documentclass[amsthm]{elsart}

\usepackage{amssymb}
\usepackage{amsmath}
\usepackage{mathrsfs}
\usepackage{graphics}

\begin{document}

\begin{frontmatter}

\title{On the configurations of nine points on a cubic curve}

\author{Alessandro Logar\thanksref{FRA}}
\address{Dipartimento di Matematica e Geoscienze,
  Universit\`a degli Studi di Trieste, Via Valerio 12/1, 34127 Trieste, Italy.}
\ead{logar@units.it}
\author{Sara Paronitti}
  \address{Dipartimento di Matematica e Geoscienze,
  Universit\`a degli Studi di Trieste, Via Valerio 12/1, 34127 Trieste, Italy.}
\ead{sara.paronitti@gmail.com}

\thanks[FRA]{Partially supported by the FRA 2018 grant
  ``Aspetti geometrici, topologici e computazionali delle variet\`{a}'',
    Universit\`{a} di Trieste}

\begin{abstract}
We study the reciprocal position of nine points in the
  plane, according to their collinearities. In particular, we consider the case
  in which the nine points are contained in an irreducible cubic curve and
  we give their classification.
  If we consider two configurations different when the associated
  incidence structures are not isomorphic, we see that there are 131
  configurations that can be realized in $\mathbb{P}^2_{\mathbb{Q}}$,
  and there are two more in  $\mathbb{P}^2_K$, where
  $K =\mathbb{Q}[\sqrt{-3}]$ (one of
  the two is the Hesse configuration given by the nine inflection points of
  a cubic curve). Finally, we compute the possible Hilbert functions
  of the ideals of the nine points. 
\end{abstract}
\end{frontmatter}

\section{Introduction} 

The study of reciprocal position of a finite number of points in
the plane and a finite set of lines joining some of the points, 
is a classical problem whose origin dates back to the past and 
that has been considered from many different points
of view. Recall, for instance, the Pappus configuration
which takes its name from Pappus of Alexandria, or
the Sylvester problem (see~\cite{wiki2}), formulated
by Sylvester in the last decade of the nineteenth century,
or the Orchard Planting Problem (see~\cite{bgs}), which
takes its origin from the book ``Rational Amusement for Winter Evenings''
by John Jackson (1821) (see also~\cite{wiki1}).

In general, by $(p_\gamma,  l_\pi)$ it is usually denoted a configuration
of $p$ points and $l$ lines such that each point is contained in $\gamma$
lines and each line contains $\pi$ points (hence $p \gamma = l\pi$). The
configurations $(p_\gamma,  l_\pi)$ have been intensively studied with
several different tools (they can be seen, for instance,
as an application of matroid
theory or of graph theory); for a more complete survey of the results we refer
to~\cite{bs,gr,cx,cx1,cx2,cx3} and the references given there.

Consider now the well known \emph{Hesse configuration}. It is realized by
the nine inflection points of a smooth cubic curve, hence it is quite natural
to generalize the problem and look for cubic curves which contain
points with some other 
collinearity conditions. For instance, in~\cite{cp} cubic curves somehow
associated to triangles are studied, while in~\cite{mpw2} it is raised
the question
whether a plane cubic curve contains the Pappus or the Desargues configuration.

In this paper we focus our attention to the case of
$9$ points in the plane (like in the
Pappus or the Hesse configurations) and we consider
the only constrain that at most triplets of points are collinear.
Then we associate to such
$9$-tuple of points the corresponding incidence
structure (two configurations of points are considered equivalent
if the corresponding incidence structures
are isomorphic). First of all we shortly classify all the
possible configurations and we
determine realizability in the projective plane $\mathbb{P}^2_K$,
where $K$ is $\mathbb{Q}$
or a suitable algebraic extension of $\mathbb{Q}$. 
Successively, we add the condition that the points lay on an
irreducible cubic curve and we
determine which configurations survive. We get in this way a complete
list of $9$-tuple
of points on a cubic curve that satisfy all the possible
kind of collinearities (the total number is $131$ if we consider the points
with coordinates in $\mathbb{Q}$, while if we extend the coordinates in
$\mathbb{Q}[\sqrt{-3}]$ the two well known M\"obius-Kantor and Hesse
configurations are added).
We see that there are several $9$-tuple of points which are realizable
in the plane,
but that, when considered on a cubic either are not realizable or
have to satisfy
further collinearities. We show that all these cases are consequence
of the Cayley-Bacharach theorem.

In Section 2 we give an algorithm which computes all the possible
incidence structures
(up to isomorphism) that can be obtained from the possible
collinearities of $9$ points
and we sketch how to determine if they are realizable in projective
planes. In Section 3
we give the final classification of those incidence structures which
lay on a cubic curve;
finally, in Section~4, we determine the possible Hilbert functions
of all the sets of points of the incidence structures of Section~3.

In order to make the computations, we have intensively used the
computer algebra
packages~\cite{sage} and~\cite{CoCoA-5}.

\section{Notations and search for possible configurations}

Let $P_0, \dots, P_8$ be $9$ points of the plane. First of all, we want to
find the possible configurations they can assume, according to
the constrain that there can be triplets but not quadruplets of
collinear points. We can represent a configuration of points
by an incidence structure (see, for instance,~\cite{mo}), i.e.\ by
a couple of sets $(\mathscr{B},\mathscr{P})$, where $\mathscr{P}$
is the set of the points $P_0, \dots, P_8$ and 
$\mathscr{B}$ is a set of blocks, where each block contains a triplet of
collinear points. For brevity, the block $(P_i, P_j, P_k)$ will be denoted
by $(i, j, k)$ or, in a more concise way, by $ijk$. Two configurations
of points represented by
the incidence structures
$(\mathscr{B},\mathscr{P})$ and $(\mathscr{B}',\mathscr{P}')$
will be considered equivalent if the two incidence structures
are \emph{isomorphic}, i.e.\ if
there exists a bijection $f : \mathscr{P} \longrightarrow \mathscr{P}'$
such that $\mathscr{B}' = \{f(B) \mid B \in \mathscr{B} \}$ (where
$f(B) = (f(b), b\in B)$). Sometimes it will be convenient to denote our
incidence structures by simply the set of blocks $\mathscr{B}$ (the
points can be deduced from the elements of the blocks). For instance,
the incidence structure $\mathscr{B} = \{(0, 1, 2)\}$ represents the
configuration of points in which $P_0, P_1, P_2$ are collinear
(and are the only triplet of collinear points) and it is isomorphic to
the incidence
structure given by $\mathscr{B}' = \{(3, 4, 7)\}$.

The cardinality $l$ of the set $\mathscr{B}$ will be called the \emph{level}
of the incidence structure.
Given an incidence structure $\mathscr{B}$ of level $l$, to obtain a new 
incidence structure describing possible collinearities of the points of
level $l+1$ it suffices to add to $\mathscr{B}$ a block
$(i, j, k)$ with the precaution though that $(i, j, k)$ has at most one
element in common with the blocks of $\mathscr{B}$ (otherwise we would
have more than $3$ collinear points in the configuration of $P_0, \dots, P_8$).
This remark allows to obtain the following algorithm (see also~\cite{blt})
which finds all the
possible configurations $\mathcal{C}_l$ of 9 points of the plane of level
$l$ for all possible levels, in which at most triplets of points are collinear. 
Since with 9 points we can form at most $b=\binom{9}{3}=84$ triplets,
termination of the algorithm is guaranteed.

\smallskip
\noindent 
\textbf{Algorithm}{ Construction of incidence structures of each level}.\\
   \texttt{Output:} A list $\mathcal{C}_1, \mathcal{C}_2, \dots, $ such that
   $\mathcal{C}_l$ is the list of all incidence structures (up to isomorphisms)
   of 9 points in the plane in which there are $l$ triplets of collinear points
   (but no 4 collinear points).\smallskip\\
   Let $\mathcal{C}_1= [\left\{(0,1,2)\right\}]$ \\
   Let $T=\left\{(i,j,k) \mid 0\leq i \leq 6,\ i+1\leq j \leq 7,\ j+1 \leq k
   \leq 8\right\}$ be the set of all the triplets that can be formed
   with $9$ points.\\
   For each $l=1,2,\dots,$ do\\
   \null\quad Set $\mathcal{C}_{l+1} = [\ ]$\\
   \null\quad For each $\mathscr{B}\in \mathcal{C}_l$ do\\
   \null\quad\quad For each $\tau\in T$ do\\
   \null\quad \quad \quad If $\tau$ has at most one point in common with each element of
   $\mathscr{B}$ then\\
   \null\quad \quad \quad \quad Let $\mathscr{B}' = \mathscr{B} \cup \{\tau\}$\\
   \null\quad \quad \quad \quad If $\mathscr{B}'$ is not isomorphic to any element
   of $\mathcal{C}_{l+1}$ then\\
   \null\quad \quad \quad \quad \quad Add $\mathscr{B}'$ to $\mathcal{C}_{l+1}$\\
   \null\quad If $\mathcal{C}_{l+1} = [\ ]$ then\\
   \null\quad \quad Set $m = l$\\
   \null\quad \quad Return the list $\mathcal{C}_l, l = 1, 2, \dots, m$.
   \medskip

 \begin{rem} The main loop, as said, can be repeated at most $84$ times and from this we
   get the termination of the algorithm, however, as soon as we have
   that from the list
   of incidence structures $\mathcal{C}_l$ we do not obtain any
   incidence structure of level
   $l+1$, the algorithm has produced all the possible incidence structures and can stop.
   In particular, in our case we get that the main loop stops with the value $m = 12$.
 \end{rem}

 The results of the algorithm are summarized by the first two lines of
 table~\ref{figFin2}. The first values of $\mathcal{C}_l$ are
 (see figure~\ref{fig2}).
\begin{eqnarray*}
  \mathcal{C}_1 &=& \left[\{012\}\right] \\
  \mathcal{C}_2 &=& \left[\{012, 345\},\{012, 034\}\right]\\
  \mathcal{C}_3 &=& \left[\{012, 034, 056\}, \{012, 034,
    135\}, \{012, 034, 156\}, \{012, 034, 567\},\right.\\
    & & \left. \{012, 345, 678\}\right] 
\end{eqnarray*}

\begin{figure}
\begin{center}
  \resizebox{0.7\textwidth}{!}{
   \includegraphics{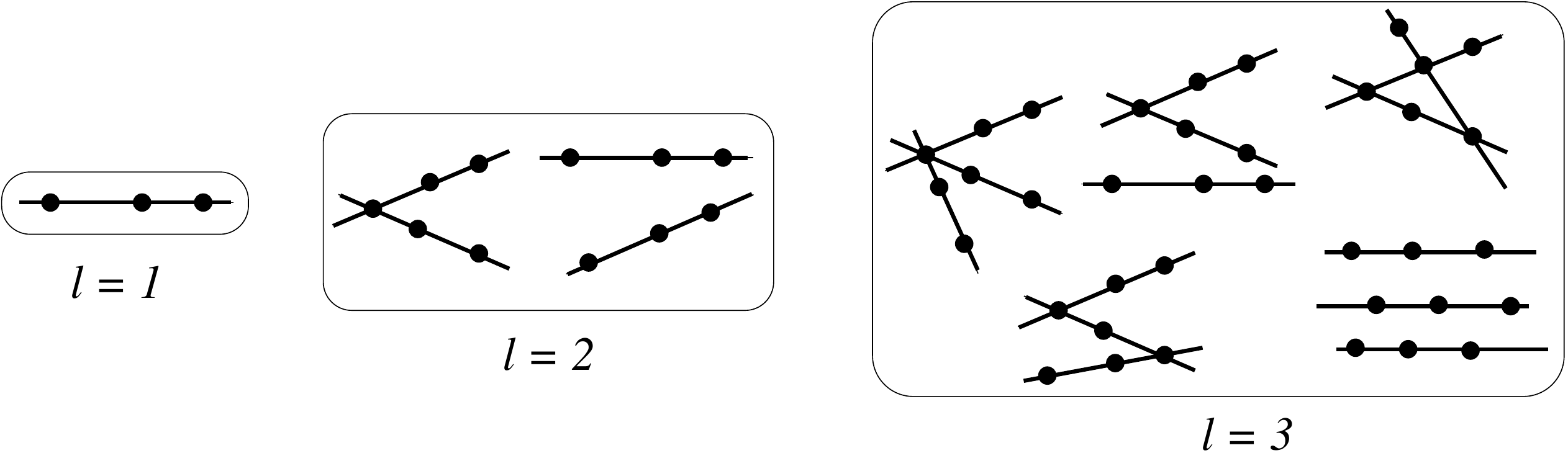}
  }
\end{center}
\caption{First three levels of incidence structures}
\label{fig2}
\end{figure}

\section{Realizability of the configurations in the plane}
We sketch here how to see which of the several incidence structures can be realized in the
projective plane $\mathbb{P}^2_K$ ($K$ a suitable algebraic
extension of $\mathbb{Q}$). As it is clear from figure~\ref{fig2},
the configurations of level $l = 1, 2, 3$ can surely be obtained in $\mathbb{P}^2_\mathbb{Q}$.
It turns out that all the other incidence structures representing the possible
configurations of points contain the blocks $(0, 1, 2), (0, 3, 4)$. Hence the
points $P_1, P_2, P_3, P_4$ are in general position and, up to a projective
motion, their coordinates can be fixed (therefore also the coordinates of $P_0$ are determined).
In particular, we can choose for $P_0, \dots, P_4$ the following points:
\begin{equation}
  P_0=(0\!:0\!:1),\ P_1=(1\!:0\!:1),\ P_2=(2\!:0\!:1),\ P_3=(0\!:1\!:1),\
  P_4=(0\!:2\!:1)
\label{5punti}
\end{equation}
We can now assign coordinates to the remaining points (depending on new variables
$t_0, t_1, \dots$). Any collinearity $(i, j, k)$ among three points $P_i, P_j, P_k$
can now be converted into the equation
given by imposing that the determinant of the matrix whose rows are the coordinates
of $P_i, P_j, P_k$, is zero. In this way we get an ideal $I$ in the polynomial ring
$\mathbb{Q}[t_0, t_1, \dots]$ and we have to study its zeros. Of course, in order to
make computations, it is quite important to keep the number of new variables as
reduced as possible. For instance, if we have the collinearity $(1, 3, 5)$, the point
$P_5$ can be expressed as $P_1+tP_3$ and its coordinates depend only on one new
variable $t$ (and $t$ is not zero, since we consider only distinct points). 
\begin{rem}
  In order to speed up the computations, to prove that a configuration is realizable
  in $\mathbb{P}^2_{\mathbb{Q}}$, it
  suffices to give random values to some of the variables and check if
  nevertheless a solution can be found. If this happens, we have avoided to consider the
  general case that could require cumbersome computations. Of course, the cases in
  which the configuration is not realizable in the plane or requires an algebraic
  extension of the field $\mathbb{Q}$, cannot be treated in this way.
  \label{oss1}
\end{rem}
It holds:
\begin{thm}
  Among the configurations given in the second line of table~\ref{figFin2} we have that the 
  configurations represented by the following blocks are not realizable:
  \begin{itemize}
  \item Level 7: $\mathscr{B}_0 = \{012, 034, 056, 135, 146, 236, 245\}$
    (the ``Fano plane $7_3$'');
    \item Level 8: $\mathscr{B}_0 \cup \{078\}$;
    \item Level 10: $\{012, 034, 056, 078, 135, 146, 237, 248, 368, 457\}$.
  \end{itemize}
  The following configurations can be realized but one (or more) further collinearities
  appear (written in bold):
  \begin{itemize}
  \item Level 8: $\{012, 034, 056, 137, 158, 248, 267, 368\}, $
    further collinearity: $\mathbf{457}$;
  \item Level 9: $\mathscr{B}_1 = \{012, 034, 056, 078, 135, 147, 168, 238, 246\}$,
    further collinearity: $\mathbf{257}$;
  \item Level 10: $\mathscr{B}_1 \cup \{367\}$, further collinearities:
    $\mathbf{257, 458}$;
  \item Level 11: $\mathscr{B}_1 \cup \{367, 257\}$, further collinearity:
    $\mathbf{458}$;
  \end{itemize}
  Finally, the following configurations are realizable but
  in $\mathbb{P}^2_K$, where $K$ is an algebraic extension of $\mathbb{Q}$
  (written on the right):
  \begin{itemize}
  \item Level 8:
    $\mathscr{B}_2 = \{012, 034, 056, 135, 147, 246, 257, 367\}$, 
    ($K=\mathbb{Q}[\sqrt{-3}]$);
  \item Level 9:
    $\{012, 034, 056, 078, 135, 147, 168, 367, 458\}$, 
    ($K=\mathbb{Q}[\sqrt{-3}]$);
  \item Level 10:
    $\{012, 034, 056, 078, 135, 146, 237, 258, 368, 457\}$,
    ($K=\mathbb{Q}[\sqrt{-1}]$);
  \item Level 12: $\mathscr{B}_2 \cup \{078, 168, 238, 458\}$,
    ($K=\mathbb{Q}[\sqrt{-3}]$).
  \end{itemize}
  All the other configurations are realizable in $\mathbb{P}^2_{\mathbb{Q}}$.
\end{thm}
\begin{proof} As soon as we have assigned coordinates to the points, the proof is
  quite straightforward. Here we prove that the configuration
  given by $\mathscr{B}_0$ is not realizable.
  Since we have the collinearity $(1, 3, 5)$, we get that
  $P_5 = P_1+t_0P_3 =(1\!: t_0\!: t_0 + 1)$. Analogously, from the collinearity
  $(2, 3, 6)$, we get $P_6 = (2\!: t_1\!: t_1 + 1)$. The collinearity $(0, 5, 6)$ is satisfied
  if $2t_0 - t_1=0$, $(1, 4, 6)$ holds if $t_1 - 2=0$ and $(2, 4, 5)$ holds if $t_0+1=0$.
  The alignment ideal associated to the blocks of $\mathscr{B}_0$ is therefore
  $I = (2t_0 - t_1, t_1 - 2, t_0+1)$ whose Gr\"obner basis is $\{1\}$, hence the configuration
  is not realizable in $\mathbb{P}^2_K$ where $K$ is any field of characteristic zero.
  Let us see the case given by the incidence structure of level 10 whose blocks are the set
  $\mathscr{B}_3 = \mathscr{B}_1\cup \{367\}$. Here the coordinates of the
  points can be the following:
  $ P_5=(1\!: t_0\!: t_0 + 1)$, $P_6 = (2\!: 2t_1\!: t_1 + 1)$,
  $P_7 = (1\!: 2t_2\!: t_2 + 1)$,
  $P_8 =  (2\!: t_3\!: t_3 + 1)$, hence the alignment ideal is
  $I = (t_3^2 - 2t_3 + 4, t_0 - 1/2t_3 + 1, t_1 - 1/2t_3 + 1, t_2 - 1/4t_3)$. The block
  $(2, 5, 7)$ gives the condition $2t_0t_2 - t_0 + 2t_2$ and the block $(4, 5, 8)$ gives
  the condition $2t_0 - t_3 + 2$ and both are in the ideal $I$, this shows that when the
  configuration given by $\mathscr{B}_3$ is realized in a plane $P^2_K$, then necessarily
  there are two further collinearities. The other cases can be solved in a similar way,
  although many computations can be simplified thanks to remark~\ref{oss1}.
\end{proof}
As a consequence of the theorem, we can complete table~\ref{figFin2} with the third
line.

\begin{rem}
An alternative proof can be obtained with the techniques developed in~\cite{bs}.
\end{rem}

\section{Points on a cubic curve}
In this Section we address the problem of determining which of the
configurations of
points described in the third line of table~\ref{figFin2}
can stay on an irreducible cubic curve of the plane.

Given a configuration of a $9$-tuple of points represented by
an incidence structure $(\mathscr{B}, \mathscr{P})$, it is often easy to see
if the nine points of $\mathscr{P}$ can lay on an irreducible cubic curve:
it is enough to take the parametrization of the points
and the alignment ideal constructed in the
previous section, to give some specific values
to the parameters in such a way that the prefigured
collinearities are satisfied (and no more appear) and finally to compute the
cubic curve through the nine points. If it exists and is irreducible,
we are done. There are however several exceptions
that can happen, it is indeed possible that in the set of blocks $\mathscr{B}$
there are three blocks that
contain all the nine points. In this case a cubic passing through the nine points
is given by the three
lines through the three blocks, hence is reducible. Therefore, in a situation like
this, we have to see if there exists also an irreducible cubic or if it is possible to
find a specific value of the parameters such that the points are contained in
an irreducible cubic curve.
We distinguish the following cases:

\begin{enumerate}
\renewcommand*{\theenumi}{\textnormal{(\alph{enumi})}}
\renewcommand*{\labelenumi}{\theenumi}
\item\label{caso1} It is possible to find an irreducible cubic curve through
  the points which satisfy the prescribed alignment.
\item\label{caso2}
  If an irreducible cubic curve contains the points, then necessarily other
  collinearities among the points appear, so the configuration belongs to a higher
  level.
\item\label{caso3}
  There does not exist an irreducible cubic curve passing through the points,
  so the configuration has to be discarded.
\end{enumerate}

The following well known result (a particularization of the Cayley-Bacharach theorem)
will be a very useful tool:

\begin{prop}
  If two cubic curves intersect in $9$ distinct points and there is a conic
  passing through $6$ of them, then the remaining $3$ points are collinear.
\label{prop:Noether}
\end{prop}
\begin{proof}
See \cite{fu}, proposition $2$ of chapter $5$.
\end{proof}

As a consequence, we have:

\begin{lem}
  Let $(\mathscr{B}, \mathscr{P})$ be an incidence structure such that $\mathscr{B}$ contains
  three blocks $i_1j_1k_1$, $i_2j_2k_2$, $i_3j_3k_3$ containing all the
  nine points $P_0, \dots, P_8$ and suppose that $\mathscr{B}$ contains two other blocks
  $B_1 =l_1m_1n_1$ and $B_2=l_2m_2n_2$ which are disjoint. Let 
  $B_3 = l_3m_3n_3$ be the triplet given by the points which are not in $B_1\cup B_2$.
  Then, if $B_3 \not \in \mathscr{B}$, it is not possible to have an irreducible
  cubic curve passing through the nine points satisfying only the alignments
  of $\mathscr{B}$. Moreover, in this hypothesis, we can distinguish two cases:
  \begin{enumerate}
  \item If $B_3$ has more than one element in common with some of the blocks of $\mathscr{B}$,
    then there are no irreducible cubics through the points;
  \item If $B_3$ has at most one element in common with the blocks of $\mathscr{B}$,
    then an irreducible cubic through the points may exist, but the points have to
    satisfy also the collinearity given by $B_3$.
  \end{enumerate}
  \label{lemma1}
\end{lem}
\begin{proof}
  Let $C_0$ be the cubic curve which splits into the three lines passing
  through the points
  which contain, respectively, $\{P_{i_1},  P_{j_1}, P_{k_1}\}$,
  $\{P_{i_2},  P_{j_2}, P_{k_2}\}$,
  $\{P_{i_3},$  $P_{j_3},$ $P_{k_3}\}$, and suppose $C$ is an irreducible
  cubic through the
  nine points. Let $D$ be the conic which splits into the lines which contain
  $\{P_{l_1},  P_{m_1}, P_{n_1}\}$ and  $\{P_{l_2},  P_{m_2}, P_{n_2}\}$.
  Then, by Proposition~\ref{prop:Noether},
  the points $\{P_{l_3},  P_{m_3}, P_{n_3}\}$ are collinear. If, moreover,
  there is a block      
  $B \in \mathscr{B}$ which has more than one point in common
  with $B_3$, then among the
  points $P_i$'s there are 4 collinear points on $C$, which is impossible.
\end{proof}

Lemma~\ref{lemma1} allows to rule out several configurations:
\begin{exmp}
Consider the incidence structure of level $7$:
\[
\mathscr{B} =\{012, 034, 056, 135, 147, 238, 267\}
\]
which is realizable in $\mathbb{P}^2_{\mathbb{Q}}$.
Here $056, 147, 238$ are three blocks of $\mathscr{B}$ which contain
all the $9$ points. From the blocks $034, 267$ we get that also the points
$P_1, P_5, P_8$ must be collinear, but the block $158$ has two points in common with
the block $135 \in \mathscr{B}$, so the four points $P_1, P_3, P_5, P_8$ should
be collinear, which is not possible: the above configuration has to be discarded.\\
Consider now the incidence structure (again of level $7$):
\[
\mathscr{B} =\{012, 034, 056, 137, 158, 248, 368\}
\]
Here $056, 137, 248$ give a reducible cubic. {From} the conic given by
$012, 368$ we see that the block $457$ has to be added to $\mathscr{B}$;
similarly, the two blocks $034, 158$ give the new
block $267$. Therefore, if nine points are on an irreducible cubic curve $C$
and satisfy the collinearities given by $\mathscr{B}$, then necessarily the
points have also to satisfy the collinearities $457$ and $267$.
So far, we do not know yet that $C$ exists, but we do know that the set of blocks $\mathscr{B}$
cannot be considered in level $7$. We can however verify that the incidence structure
$\mathscr{B} \cup \{267, 457\}$ is isomorphic to an incidence structure of level 9.
\end{exmp}

Lemma~\ref{lemma1} allows to find cases of type~\ref{caso3}
and to have candidates for cases of type~\ref{caso2} of the previous list, but we still
have to see if there are other configurations that, for some reason, cannot exist or
do not exist in the expected level.

Consider the example given by the incidence structure (of level 7):
\[
\mathscr{B} =\{012, 034, 056, 135, 146, 367, 458\}
\]
  Lemma~\ref{lemma1} does not give information, so in order to see if there exists
an irreducible cubic passing through the $9$ points of the above incidence structure,
we need a different approach. The generic points
of the corresponding configuration are $P_0, P_1, P_2, P_3, P_4$ given
in~(\ref{5punti}) and:
\[
\begin{array}{l}
  P_5= (1\!:t_0\!: t_0 + 1),\ P_6= (1\!: 2t_1\!: t_1 + 1),\
  P_7= (t_2\!: 2t_1t_2 + 1\!: t_1t_2 + t_2 + 1),\\
  P_8= (t_3\!: t_0t_3 + 2\!: t_0t_3 + t_3 + 1) 
\end{array}
\]

It is easy to verify that, with the condition $t_0-2t_1 = 0$, on the parameters
(i.e. in this
case the alignment ideal is $(t_0-2t_1)$), the points satisfy
the collinearities of the set of blocks $\mathscr{B}$.
In general, the only cubic containing these points is reducible and
comes from the blocks
$012, 367, 458$, hence we have to find conditions on the coordinates of the
points in order to have other cubic
curves. We consider the linear system of cubic curves
passing through $P_0, \dots, P_4$, which is:
\begin{equation}
  ax^3 + bx^2y + cxy^2 + (1/2)dy^3 - 3ax^2z + exyz - (3/2)dy^2z + 2axz^2 + dyz^2=0
  \label{eq:cubic}
\end{equation}
The condition that the remaining $4$ points satisfy $(\ref{eq:cubic})$ gives
a system of four linear, homogeneous equations in the variables $a, b, c, d, e$
whose coefficients are polynomials in $t_0, t_1, \dots$. In order to have other solutions,
the rank of the matrix $M$ associated to the system must be less than $4$.
Hence we collect the order $4$ minors of $M$ and we get a new ideal
(in $t_0, t_1, \dots$) that has to be added to the alignment ideal of the points.
After some manipulation (the ideal can be saturated w.r.t.\ the variables) we get the
ideal $J$:
\[
J = \left(2t_1^2t_2t_3 - 4t_1^2t_2 + 2t_1^2t_3 - 1, t_0 - 2t_1 \right)
\]
A zero of $J$ is for instance
$t_0 = 6, t_1 =3, t_2=-13/18, t_3=-5$  which gives the points:
$P_5 =  (1\!: 6\!: 7)$, $P_6 = (1\!: 6\!: 4)$, $P_7=(13\!: 60\!: 34)$,
$P_8= (5\!: 28\!: 34)$.
The collinearities satisfied by these points are precisely those of the incidence
structure $\mathscr{B}$. A cubic curve containing the $9$ points is:
\begin{equation*}
456x^3 - 78x^2y + 623xy^2 - 26y^3 - 1368x^2z - 1340xyz + 78y^2z + 912xz^2 - 52yz^2
\end{equation*}
which is irreducible (and smooth). In conclusion, concerning this case, we can find
nine points which satisfy the described alignments and that are contained on an irreducible
cubic curve. This configuration is of type~\ref{caso1}.

These kind of computations can be done for each of the cases we have to consider and,
together with lemma~\ref{lemma1}, allow to obtain the following conclusion:

\noindent
\emph{Level 1, 2, 3, 4}: there are no restrictions.  \\
\emph{Level 5}. There is one incidence structure of type~\ref{caso2} (in black the
further alignment):
\begin{itemize}
\item $012$, $034$, $156$, $278$, $357$, $\mathbf{468}$, (it belongs to level 6).
\end{itemize}

\noindent
\emph{Level 6}. There is one incidence structure of type~\ref{caso2} (in black the
further alignment):
\begin{itemize}
\item $012$, $034$, $056$, $137$, $158$, $248$, $\mathbf{267}$, (it belongs to level 7);
\end{itemize}
and one incidence structure of type~\ref{caso3}:
\begin{itemize}
\item $012$, $034$, $056$, $137$, $248$, $578$.
\end{itemize}

\noindent
\emph{Level 7}. The incidence structures of type~\ref{caso2} are (in black the further
alignments): 
\begin{itemize}
\item $012$, $034$, $056$, $078$, $135$, $147$, $238$, $\mathbf{246}$, (it belongs to level 8);
\item  $012$, $034$, $056$, $135$, $147$, $238$, $246$, $\mathbf{078}$, (it belongs to level 8);
\item  $012$, $034$, $056$, $137$, $158$, $248$, $368$, $\mathbf{267}$, $\mathbf{457}$,
  (it belongs to level 9);
\end{itemize}
The incidence structures of type~\ref{caso3} are:
\begin{itemize}
\item $012$, $034$, $056$, $135$, $147$, $238$, $267$;
\item $012$, $034$, $056$, $135$, $147$, $238$, $678$;
\item $012$, $034$, $056$, $137$, $158$, $248$, $467$.
\end{itemize}

\noindent
\emph{Level 8}. The incidence structures of type~\ref{caso2} are:
\begin{itemize}
\item $012$, $034$, $056$, $078$, $135$, $147$, $168$, $238$,  $\mathbf{246}$,
  $\mathbf{257}$, (it belongs to level 10);
\item $012$, $034$, $056$, $078$, $135$, $147$, $238$, $257$, $\mathbf{168}$,
  $\mathbf{246}$, (it belongs to level 10);
\end{itemize}
the incidence structures of type~\ref{caso3} are:
\begin{itemize}
\item  $012$, $034$, $056$, $078$, $135$, $146$, $367$, $458$;
\item $012$, $034$, $056$, $078$, $135$, $147$, $238$, $267$;
\item $012$, $034$, $056$, $135$, $146$, $278$, $367$, $458$;
\item $012$, $034$, $056$, $135$, $147$, $238$, $246$, $578$;
\item $012$, $034$, $056$, $135$, $147$, $238$, $267$, $468$.
\end{itemize}

\noindent
\emph{Level 9}. The incidence structures of type~\ref{caso2} are:
\begin{itemize}
\item $012$, $034$, $056$, $078$, $135$, $147$, $168$, $367$, $458$, 
  $\mathbf{246}$, $\mathbf{257}$, $\mathbf{238}$, (it belongs to level 12);
\item $012$, $034$, $056$, $078$, $135$, $147$, $168$, $238$, $367$,  
$\mathbf{246}$, $\mathbf{257}$, $\mathbf{458}$, (it belongs to level 12);
\end{itemize}
the incidence structures of type~\ref{caso3} are:
\begin{itemize}
\item  $012$, $034$, $056$, $078$, $135$, $146$, $237$, $368$, $457$;
\item $012$, $034$, $056$, $078$, $135$, $147$, $238$, $257$, $468$;
\item $012$, $034$, $056$, $135$, $147$, $238$, $267$, $468$, $578$;
\end{itemize}

\noindent
\emph{Level 10}. The incidence structure of type~\ref{caso3} is:
\begin{itemize}
\item $012$, $034$, $056$, $078$, $135$, $146$, $237$, $258$, $368$, $457$;
\end{itemize}
  
\bigskip
In particular, we have the following conclusion:
\begin{thm}
  Let $P_0, \dots, P_8$ be $9$ points in the projective plane $\mathbb{P}_{\mathbb{Q}}^2$
  laying on an irreducible
  cubic curve. If we distinguish the configurations of the points according to the
  possible collinearities
  satisfied by them (and we consider two configurations different if and only if the
  corresponding incidence structures are not isomorphic), we have
  that the number of possibilities is given by the last line of table~\ref{figFin2}.
  \label{teorema133}
\end{thm}

\begin{table}
\begin{tabular}{|c|c|c|c|c|c|c|c|c|c|c|c|c|} \hline
  Level &  1 & 2 & 3 & 4 & 5 & 6 & 7 & 8 & 9 & 10 & 11 & 12  \\ \hline
  \# inc.\ struct. & 1 & 2 & 5 & 11 & 19 & 34 & 41 & 31 & 12 & 4 & 1 & 1 \\ \hline
            \# realiz.\ conf. & 1 & 2 & 5 & 11 & 19 & 34 & 40 & 29${}^*$ & 11 & 2${}^{**}$ &
                                         0 & 1${}^*$ \\ \hline
 \# realiz.\ conf. on a cubic & 1 & 2 & 5 & 11 & 18 & 32 & 34 & 22${}^*$ & 6 & 1 &
                                         0 &  1${}^*$ \\ \hline
\end{tabular}\\
${}^*$ one to be realized needs the field $K = \mathbb{Q}[\sqrt{-3}]$\\
${}^{**}$ one to be realized needs the field $K = \mathbb{Q}[\sqrt{-1}]$\\
\caption{Line 2: Number of incidence structures of $9$ points in the plane;
  Line 3: number of realizable configurations of $9$ points in the plane;
  Line 4: number of configurations of $9$ points on a cubic curve.}
\label{figFin2}
\end{table}

\begin{rem} To complete the information of table~\ref{figFin2}, we
  add here that the configuration $012, 034, 156, 278, 357$ of level $5$
  which belongs to level $6$ when considered on a cubic curve (with the additional
  collinearity $468$) appears in~\cite{pm2}, page 50 and in~\cite{pm};
  the configuration of level~8 which needs $\mathbb{Q}[\sqrt{-3}]$
  to be realized is: 
  $012$, $034$, $056$, $135$, $147$, $246$, $257$, $367$ (this configuration is
  also known as the M\"obius-Kantor configuration),  
  the configuration
  of level 10 which needs $\mathbb{Q}[\sqrt{-1}]$ is:
  $012$, $034$, $056$, $078$, $135$, $146$, $237$, $258$, $368$, $457$ 
  and, as
  said above, it is of type~\ref{caso3} and cannot be contained in an irreducible
  cubic curve. Finally the configuration of level 12 (realizable in $\mathbb{Q}[\sqrt{-3}]$)
  is the Hesse configuration:
  $012$, $034$, $056$, $078$, $135$, $147$, $168$, $238$, $246$, $257$, $367$, $458$ 
  and is an extension
  of the above configuration of level 8.\\
  In figure~\ref{figuraC8} we show the example of the configuration
  $012$, $034$, $056$, $078$, $135$, $147$, $168$, $238$ of level $8$, introduced in
  the level $8$ of the previous list, 
  which can be realized in the plane, but that can be contained in an irreducible cubic
  curve only if point $5$ (of figure~\ref{fig5})
  coincides with point $5'$ and point $6$ coincides with
  point $6'$; this gives the only configuration  of level $10$ contained in a cubic curve
  (and contains the well known Pappus configuration).\\
  Note that, as a consequence of the above computations, we see that, when some
  collinearities of $9$ points on a cubic curve force some other collinearities, this
  can always be explained as an application of proposition~\ref{prop:Noether}.
\end{rem}

\begin{figure}
\begin{center}
\resizebox{4cm}{!}{
\includegraphics{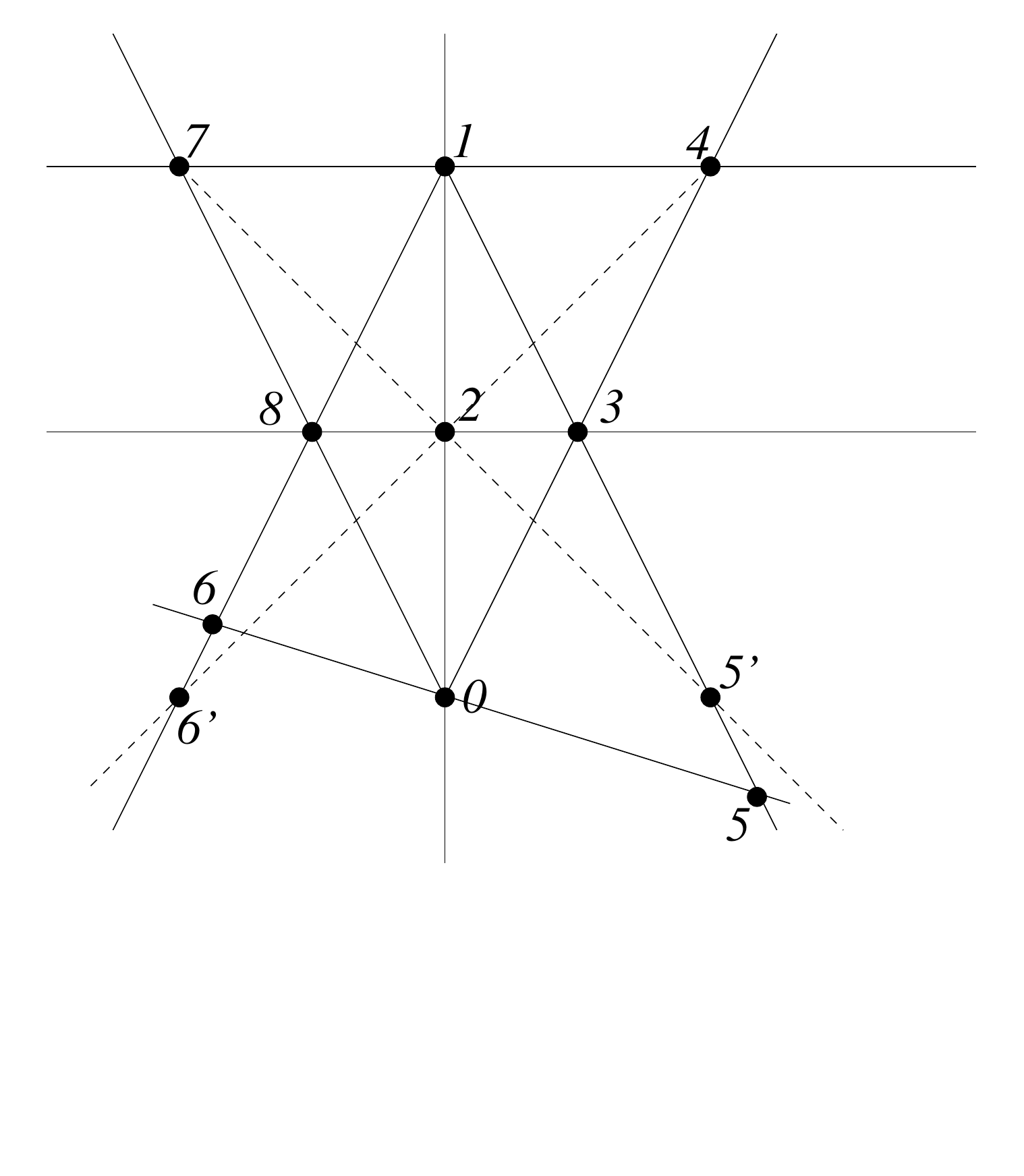}
}
\resizebox{5cm}{!}{
  \includegraphics{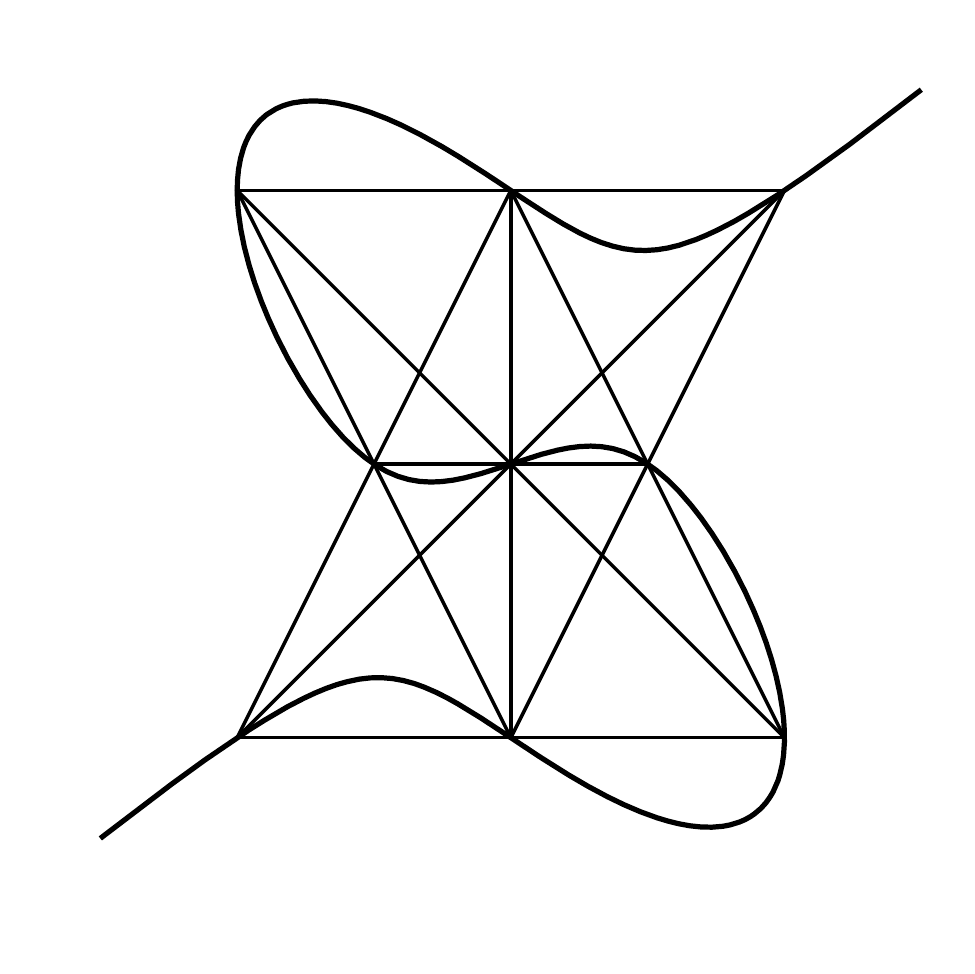}
}
\end{center}
\caption{The configuration $012$, $034$, $056$, $078$, $135$, $147$, $168$, $238$ 
  of level $8$ which contains two further collinearities when
  constrained on a cubic curve.}
\label{fig5}
\label{figuraC8}
\end{figure}

\section{The Hilbert functions of the nine points}
In the previous Section we classified points on an irreducible cubic curve,
according to the possible collinearities. It is therefore quite natural
to ask about the Hilbert functions of the ideal of the nine points of each
configuration. First, observe that the linear system of
nine points on an irreducible cubic curve
has dimension either $0$ or $1$ (see~\cite{tg} and~\cite{lu});
hence, if $I$ is the ideal of the nine points $P_0, \dots, P_8$, then
$\dim \left(R/I\right)_3 = 9$ or
$\dim \left(R/I)\right)_3 = 8$, where $R = K[x, y, z]$. 
Suppose $F \in R$ is the generic homogeneous polynomial of
degree $d \geq 3$, with coefficients $a_0, a_1, \dots, a_{m-1}$, 
$m= {d+2\choose d}$, and let $M(d, P)$ be the $9 \times m$
matrix whose rows are the coefficients of $a_0, a_1, \dots$ in $F(P_i)$
($i = 0, \dots, 8$). The matrix $M(d, P)$ has rank at most $9$. 
\begin{lem} If $M(d, P)$ has rank $9$, then also $M(d+1, P)$ has rank $9$.
  \label{lemma51}
\end{lem}
\begin{proof}
  We can assume that (after a change of coordinates, if necessary)
  all the points $P_i$ have the last coordinate different from zero
  and hence can be chosen equal to $1$. Let
  $F = a_0x^d+a_1x^{d-1}y+a_2x^{d-1}z+\cdots$, then the generic polynomial
  of degree $d+1$ can be seen as $F\cdot z + H(x, y)$, where $H$
  is the generic homogeneous polynomial of degree $d+1$ in $x, y$. With
  these notations, the first $m$ columns of $M(d+1, P)$ are
  the matrix $M(d, P)$. {}From this, the result follows.
\end{proof}
This lemma allows to obtain the following:
\begin{thm}
  The Hilbert function of $R/I$, where $I$ is the ideal of points
  of the configurations described in
  theorem~\ref{teorema133} is either $0 \mapsto 1, 1\mapsto 3, 2 \mapsto 6,
  3\mapsto 9, 4 \mapsto 9, \dots$, if there
  exists only one irreducible cubic containing the points, or
  $0 \mapsto 1, 1\mapsto 3, 2 \mapsto 6,$
  $3\mapsto 8,$ $ 4 \mapsto 9,$ $ 5 \mapsto 9,  \dots$
  in the other case, when
  the linear system of the cubic curves
  through the points is of dimension $1$. 
\end{thm}
\begin{proof}
  If the nine points are contained in only one cubic curve, the result
  immediately follows from lemma~\ref{lemma51}. If we have two irreducible
  cubic curves through the nine points, the result is a consequence
  of the fact that the points are a complete intersection.
\end{proof}

It is possible to verify that, when the level of a configuration
is $5$ or more, only one of the
two Hilbert functions is possible for that configuration,
meanwhile, it can happen that a configuration of level $4$ or less admits
both Hilbert functions:
\begin{exmp} Consider the configuration of level 4: $012, 034, 056, 078$.
  In general, the linear system of $9$ points in this configuration is
  of dimension $0$ but there are suitable positions of the points which
  are contained in two irreducible cubic curves, like the following:
  $P_0 = (0\!: 0\!: 1)$, $P_1 = (1\!: 0\!: 1)$, $P_2 = (2\!: 0\!: 1)$,
  $P_3 = (0\!: 1\!: 1)$,
  $P_4 =  (0\!: 2\!: 1)$, $P_5 = (2\!: -3\!: 1)$,
  $P_6 =  (4\!: -6\!: -5)$,
 $P_7 =  (30\!: -19\!: 9)$, $P_8 =  (60\!: -38\!: 15)$.
\end{exmp}

\bibliographystyle{plain}
\bibliography{paperLP}

\end{document}